\documentclass[12pt]{article}%
\usepackage{amsfonts,color}
\usepackage{amsmath}
\usepackage{amssymb}
\usepackage{graphicx}%
\providecommand{\U}[1]{\protect\rule{.1in}{.1in}}
\newtheorem{theorem}{Theorem}

\newtheorem{corollary}[theorem]{Corollary}

\newtheorem{definition}[theorem]{Definition}

\newtheorem{proposition}[theorem]{Proposition}
\newtheorem{remark}[theorem]{Remark}

\newenvironment{proof}[1][Proof]{\noindent\textbf{#1.} }{\ \rule{0.5em}{0.5em}}

\title{Some inclusion results for interpolated summing operator ideals and integrability improvement of vector valued functions}
\author{{\bf D. Pellegrino} \\ {\small Departamento de Matem\'atica}\\  {\small Universidade
Federal da Para\'iba} \\  {\small 58.051-900, Jo\~ao Pessoa, Brazil} \\  {\small e-mail: dmpellegrino@gmail.com}\\ {\bf P. Rueda}\\   {\small  Departamento de An\'alisis Matem\'atico} \\
 {\small  Universidad de Valencia}\\ {\small 46100 Burjassot - Valencia, Spain}\\   {\small e-mail: pilar.rueda@uv.es}\\  {\bf E.A. S\'anchez-P\'erez} \\   {\small Instituto Universitario de
Matem\'{a}tica Pura y Aplicada}\\   {\small Universitat Polit\`ecnica de Val\`encia} \\  {\small Camino
de Vera s/n, 46022 Valencia, Spain} \\ {\small  e-mail: easancpe@mat.upv.es}}

\begin{document}

\maketitle

\begin{abstract}
Consider a Banach space valued measurable function $f$ and an operator $u$ from the space where {$f$} takes values. If $f $ is Pettis  integrable, a classical result due to
 J. Diestel shows that composing it  with $u$ gives a Bochner integrable function $u \circ f$ whenever $u$ is absolutely summing. In a previous work we have shown that a well-known
interpolation technique for  operator ideals allows to prove under some requirements that a composition of a $p$-Pettis integrable function with a $q$-summing operator provides an $r$-Bochner
 integrable function. In this paper a new abstract inclusion theorem for classes of {abstract} summing operators is shown and  applied to the class of interpolated operator ideals. Together with the results of the {aforementioned} paper, it
provides more results on the relation about the integrability of the function $u \circ f$ and the summability properties of $u$.
\end{abstract}


\section{Introduction}
The fact that absolutely summing operators improve the summability of sequences lies in the core of the theory. But the classical Diestel theorem shows that absolutely summing operators improve not only the summability of sequences, but also the integrability of vector valued functions: 
an operator $u:X \to Y$ is absolutely summing
 if and only if it transforms each Pettis integrable function to a Bochner integrable function when composed with $u$ (see \cite{Di}).

The technique for proving this theorem cannot be {transfered} directly to { general $q$-summing operators in order to give }  characterizations of operators transforming $p$-Pettis integrable functions to $r$-Bochner integrable functions in terms of their summability properties.
Recently, we have published a paper in which a different procedure is used {to get} such results (see \cite{perusan}). In {that paper}, an interpolation procedure for operator ideals (see \cite{Matt87,LS93,LS00}) is used as well as  new vector valued function spaces constructed by
 ``interpolating"  the $p$-Pettis and the $p$-Bochner integrable functions.

 In the present work we {use  the main result of \cite{perusan} jointly with } a new inclusion theorem for interpolated operator ideals {to find new insights of the relation between summing operators and the improvement of integrability of a vector valued function. It is worth mentioning that the inclusion theorem has been modeled in an abstract setting whose roots can be found in the abstract domination theorem proven in \cite{BoPeRu}}

\section{Basic definitions}

We refer to \cite{DF92,DJT95,Pie80} for definitions and general results on operator ideals and $p$-summing operators, and to \cite{LS00,Matt87} for the class of $(p,\sigma)$-absolutely continuous operators. {This class} {forms} {an operator ideal that is constructed by means of } an interpolation procedure.
Let $1 \le p < \infty$ and $0 \le \sigma < 1$. A (linear and continuous)  operator $u:X \to Y$ between Banach spaces is  {\it $(p, \sigma)$-absolutely continuous} if {there is a constant $C>0$ such that} for every $x_1,...,x_n \in X$,
$$
\Big( \sum_{i=1}^n \big\| u(x_i) \big\|^{\frac{p}{1-\sigma}} \Big)^{\frac{1-\sigma}{p}} \le C \sup_{x' \in B_{X^*}} \Big( \sum_{i=1}^n \big( | \langle x_i, x' \rangle|^{1-\sigma} \| x_i \big\|^\sigma \big)^{\frac{p}{1-\sigma}} \Big)^{\frac{1-\sigma}{p}}.
$$
The space of all linear  operators $u$ from $X$ to $Y$ satisfying these inequalities is denoted  by $\Pi_p^\sigma(X,Y)$. It is a normed space, being $\pi_p^{\sigma}(u)$
 the norm computed as the infimum of all the constants $C$  ( see \cite{LS93,LS00,Matt87} for further details).

The following class of vector valued function spaces is relevant in this paper.
Let $0 \le \sigma \le 1$ and $1 \le p < \infty$. Let $\mu$ be a finite measure. Consider the space ${\mathcal S}_{p}^\sigma(\mu,X)$ of all equivalence classes with respect to $\mu$ of simple  functions with values in the Banach space $X$.
Obviously, all such functions $f$ satisfy
$$
 \big( | \langle f(\cdot), x' \rangle |^{1-\sigma} \|f(\cdot)\|^\sigma \big)^p \in L^1(\mu)
$$
for all ${x' \in X^*},$ and
$$
\Phi_{p,\sigma}(f):= \sup_{x' \in B_{X^*}}  \Big( \int \big( | \langle f(w), x' \rangle |^{1-\sigma} \|f(w)\|^\sigma \big)^p \, d \mu  \Big)^{1/p} < \infty.
$$

A seminorm for this space can be given by the convexification $\|\cdot \|_{p,\sigma}$ of the homogeneous function $\Phi_{p,\sigma}$,
$$
\|f\|_{p,\sigma}:= \inf \big\{ \sum_{i=1}^n \Phi_{p,\sigma}(f_i): f=\sum_{i=1}^n f_i \big\},
\quad f\in {\mathcal S}_p^\sigma(\mu,X).
$$
We will write $\mathcal P_p^\sigma(\mu,X)$ for the function space associated to the
 completion  {$\overline{S_p^\sigma(\mu,X)}$ of ${\mathcal S}_{p}^\sigma(\mu,X)$, that is,
 ${\mathcal P}_p^\sigma(\mu,X)$  is the subspace of the elements of $\overline{S_p^\sigma(\mu,X)}$  that can be represented by a function in ${\mathcal P}_{p}(\mu,X)$. }
The reader can find all the required information about these spaces,
composition operators and $(p,\sigma)$-absolutely continuous operators in  \cite{perusan}. For the case $\sigma=0$, we get the classical space
$\mathcal P_p(\mu,X)$ of Pettis integrable functions that are strongly measurable.

\section{An abstract inclusion theorem}

\bigskip Let $Y$ be an arbitrary set, $X$ a vector space and let $K$ be a compact topological space.
Consider a measure space $\left(  \Omega
,\Sigma,\nu\right)  $. Let $\mathcal{H}$ be a set of functions from $\Omega$ into $X$ such that $\alpha\mathcal{H}=\mathcal{H}$ for all
${\alpha \in \mathbb R}$.

Let $\mathcal{F}$ be a set of maps from $X$ to $Y$. Consider two functions
\begin{align*}
S  & :\mathcal{F}\times\mathcal{H}\rightarrow\mathcal{L}_{q}\left(  \nu\right)
\\
R  & :\mathcal{H\times}W\rightarrow\mathcal{L}_{p}\left(  \nu\right)
\end{align*}
such that%
$$
\left\vert \alpha S\left(  u,f\right)  \right\vert   =\left\vert S\left(
u,\alpha f\right)  \right\vert
$$
and
$$
\left\vert \alpha R\left(  f,k\right)  \right\vert  =\left\vert R\left(
\alpha f,k\right)  \right\vert
$$
for all $\alpha \in \mathbb R$, $u \in \mathcal F$, {$k\in K$ and $f \in \mathcal H$. Note that if $g\in \mathcal{L}_{q}\left(  \nu\right)$ and $f\in \mathcal{H}$ then $g(w)f\in \mathcal{H}$ for all $w\in \Omega$.} In what follows we will write explicitly the variable $w$ when we want to emphasize that we are referring to the (scalar) value of $g$ at the point $w$.

\bigskip


\begin{definition}\rm
{A map $u:X\rightarrow Y$ is {\it $(  q,p)$-$RS$ summing} if there is a constant $C>0$ such that
\[
\left(
{\displaystyle\int\limits_{\Omega}}
 \big| S\left(  u,g(w)f\right)(w) \big|  ^{q}  \, d\nu\right)^{\frac{1}{q}}\leq C\sup_{k\in K} \left(
{\displaystyle\int\limits_{\Omega}}
 \big|  R ( g(w) f,k)(w)   \big| ^{p} \, d\nu \right)^{\frac{1}{p}}
\]
for all $f:\Omega\rightarrow X$ in $\mathcal{H}$ and all $g\in \mathcal L_q(\nu)$. Let $RS\left(   q ,p\right)$ denote the class of all $(  q,p)$-$RS$ summing mappings.}
\end{definition}

\bigskip

{The next result} gives our inclusion result for  $(q,p)$-$RS$ summing operators. Typically, the measure $\nu$ is atomic, $\sigma$-finite but not finite. This will provide in the next section the main tool
for {inclusions} among classes of $(q,p,\sigma)$-absolutely continuous operators.

\begin{theorem} \label{inclu}
Suppose that the measure $\nu$ satisfies that for  $1 \le r<s < \infty$ there are constants $C_{s,r}$ such that  $\left\Vert {\cdot}\right\Vert _{L_{s}(\nu)}\leq C_{s,r} \, \left\Vert \cdot
{}\right\Vert _{L_{r}(\nu)}.$ Suppose that $1\leq p_{1}\leq p_{2}<\infty$, and
$1\leq q_{1}\leq q_{2}<\infty$ and%
\[
\frac{1}{p_{1}}-\frac{1}{p_{2}}\leq\frac{1}{q_{1}}-\frac{1}{q_{2}}.
\]
Then%
\[
RS\left(   q_{1} ,p_{1}\right)  \subset RS\left(
q_{2},p_{2}\right).
\]

\end{theorem}

\begin{proof} 
{Take $u \in RS\left(q_{1} ,p_{1}\right)$. }
Let $\frac{1}{p}  =\frac{1}{p_{1}}-\frac{1}{p_{2}}$ and $\frac{1}{q}  =\frac{1}{q_{1}}-\frac{1}{q_{2}},$
and {$f\in \mathcal H$ and $g\in \mathcal L_{q_2}(\nu)$.} Define
\[
\lambda(w)=S\left(  u,g(w) f\right)(w)^{\frac{q_{2}}{q}}.
\] 
As $q_1\leq q_2$ it follows that $S(u,f)\in \mathcal L_{q_2}$ and since $\frac{1}{q_2}+\frac{1}{q}=\frac{1}{q_1}$ then $\lambda g\in \mathcal L_{q_1}$. 
Then for $\beta= q_1/p_1$ we obtain
\begin{align*}%
{\displaystyle\int\limits_{\Omega}}
\left\vert S\left(  u,g(w)f\right)(w)  \right\vert ^{q_{2}}d\nu & =%
{\displaystyle\int\limits_{\Omega}}
\left\vert S\left(  u,g(w)f\right)(w)  \right\vert ^{q_{1}}\left\vert \lambda(w)
\right\vert ^{q_{1}}d\nu\\
& =
{\displaystyle\int\limits_{\Omega}}
\left\vert S\left(  u,\lambda(w) g(w)f\right)(w)  \right\vert ^{q_{1}}d\nu\\
& \leq C\left(  \sup_{k\in K}%
{\displaystyle\int\limits_{\Omega}}
\left\vert R\left(  \lambda(w) g(w) f,k\right)(w)  \right\vert ^{p_{1}}d\nu\right)
^{\beta} \\
& =C\left(  \sup_{k\in K}%
{\displaystyle\int\limits_{\Omega}}
\left\vert S\left(  u,g(w)f\right)(w) \right\vert ^{p_{1}\frac{q_{2}}{q}} \left\vert
R\left( g(w) f,k\right)(w)  \right\vert ^{p_{1}}d\nu\right)  ^{\beta}.
\end{align*}
Since
\[
\frac{1}{\left(  \frac{p}{p_{1}}\right)  }+\frac{1}{\left(  \frac{p_{2}}%
{p_{1}}\right)  }=1,
\]
using  H\"older's inequality we obtain%
\begin{align*}
{\displaystyle\int\limits_{\Omega}}
\left\vert S\left(  u,g(w)f\right)(w)  \right\vert ^{q_{2}}d\nu
\end{align*}
\begin{align*}
\leq C\left(
\sup_{k\in K}\left(
{\displaystyle\int\limits_{\Omega}}
\left(  \left\vert S\left(  u,g(w)f\right)(w)  ^{p_{1}\frac{q_{2}}{q}}\right\vert
\right)^{\frac{p}{p_{1}}}d\nu\right)  ^{\frac{p_{1}}{p}}
\cdot \,\,
\left(
{\displaystyle\int\limits_{\Omega}}
\left(  \left\vert R\left( g(w) f,k\right)(w)  \right\vert ^{p_{1}}\right)
^{\frac{p_{2}}{p_{1}}}d\nu\right) ^{\frac{p_{1}}{p_{2}}}\right)  ^{\beta}
\end{align*}
$$
 =C\left(
{\displaystyle\int\limits_{\Omega}}
\left\vert S\left(  u,g(w)f\right)(w)  ^{\frac{q_{2}p}{q}}\right\vert d\nu\right)
^{\frac{p_{1}}{p}\beta}  \cdot \,\, \sup_{k\in K}\left(
{\displaystyle\int\limits_{\Omega}}
\left\vert R\left(  g(w)f,k\right)(w)  \right\vert ^{p_{2}}d\nu\right)  ^{\frac
{p_{1}}{p_{2}}\beta}.
$$
But since $p\geq q$ we have%
\[
{\displaystyle\int\limits_{\Omega}}
\left\vert S\left(  u,g(w)f\right)(w)  \right\vert ^{q_{2}}d\nu 
\]
\[
\leq {C (C_{p,q})^{p_1 \beta}} \left(
{\displaystyle\int\limits_{\Omega}}
\left\vert S\left(  u,g(w)f\right)(w)  \right\vert ^{q_{2}}d\nu\right)  ^{\frac
{p_{1}}{q}\beta}\sup_{k\in K}\left(
{\displaystyle\int\limits_{\Omega}}
\left\vert R\left(  g(w)f,k\right)(w)  \right\vert ^{p_{2}}d\nu\right)  ^{\frac
{p_{1}}{p_{2}}\beta}
\]
and thus%
\[
\left(
{\displaystyle\int\limits_{\Omega}}
\left\vert S\left(  u,g(w)f\right)(w)  \right\vert ^{q_{2}}d\nu\right)
^{1-\frac{p_{1}}{q}\beta}
\leq
{C(C_{p,q})^{p_1 \beta}} \,
\sup_{k\in K}\left(
{\displaystyle\int\limits_{\Omega}}
\left\vert R\left(  g(w)f,k\right)(w)  \right\vert ^{p_{2}}d\nu\right)  ^{\frac
{p_{1}}{p_{2}}\beta},
\]
i.e.,%
\[
\left(
{\displaystyle\int\limits_{\Omega}}
\left\vert S\left(  u,g(w)f\right)(w)  \right\vert ^{q_{2}}d\nu\right)
^{\frac{1}{q_2}}
\leq
{C^{1/q_1} C_{p,q} }\,
\sup_{k\in K}\left(
{\displaystyle\int\limits_{\Omega}}
\left\vert R\left(  g(w)f,w\right)(w)  \right\vert ^{p_{2}}d\nu\right)^{\frac{1}{p_2}}.
\]

\end{proof}


\section{ The inclusion theorem for $(p,q,\sigma)$-absolutely continuous operators}

As in the case of $p$-summing operators, a notion of $(p,\sigma)$-absolutely continuous operator with different indexes in the left and right hand sides of the inequality can be made,
and {it defines}  new classes of operators with particular summability properties. Let $1 \le p \le q < \infty$ and $0 \le \sigma < 1$ { and let $X$ and $Y$ be Banach spaces.}

\begin{definition}\rm
{A linear } operator $u:X \to Y$ is {\it $(q,p,\sigma)$-absolutely continuous} if there is a constant $C>0$ such that for every $x_1, ..., x_n \in X$,
\begin{equation}\label{pq}
\Big( \sum_{i=1}^n \big\| u(x_i) \big\|^{\frac{q}{1-\sigma}} \Big)^{\frac{1-\sigma}{q}} \le C \sup_{x' \in B_{X^*}} \Big( \sum_{i=1}^n \big( | \langle x_i, x' \rangle|^{1-\sigma} \| x_i \big\|^\sigma \big)^{\frac{p}{1-\sigma}}
 \Big)^{\frac{1-\sigma}{p}}.
\end{equation}
We will write $\Pi_{q,p}^\sigma(X,Y)$ for this class.
\end{definition}

{Let us see that $(q,p,\sigma)$-absolutely continuous operators are a particular class of $(q',p')$-$RS$ summing mappings.}
Consider $(\Omega, \Sigma, \nu)$ to be the natural numbers with the counting measure $(\mathbb N, \mathcal P(\mathbb N), c)$.  Consider $B_{X^*}$ with the weak* topology as {the compact set $K$}, and identify  the set of functions
$\mathcal H$ with the simple functions that can be { identified with }  finite sequences having values in $X$ {via the bijection $\chi_i\leftrightarrow e_i$, where $\chi_i$ is the characteristic function that takes the value $1$ on $i$ and $0$ otherwise, and $e_i$ is the canonical sequence with $1$ in the $i$-th coordinate and $0$ otherwise, $i\in \mathbb N$}. We consider also $\mathcal F$ to be the set of linear and continuous operators $u:X \to Y$.

For each simple function/finite sequence $f= \sum_{i=1}^n x_i \chi_{i} = \sum_{i=1}^n x_i e_{i} $,
we define
$$
S(u,f)(i):= \|u(f(i))\|, \quad i=1,\ldots,n,
$$
 and
$$
R(f,x')(i)=|\langle f(i),x' \rangle|^{1-\sigma} \|f(i)\|^\sigma, \quad i=1,\ldots,n.
$$

\begin{proposition}
{A linear operator $u:X \to Y$ is  $(q,p,\sigma)$-absolutely continuous if, and only if, it is $(\frac{q}{1-\sigma},\frac{p}{1-\sigma})$-$RS$ summing, for $R$ and $S$ given as above.}
\end{proposition}

\begin{proof}
{
 Note that for $g\in \mathcal L_q(c)=\ell_q$,
$$
\int_\mathbb N |S(u,gf)|^{\frac{q}{1-\sigma}} \, d c  = \sum_{i=1}^n \|u(g(i)x_i)\|^{\frac{q}{1-\sigma}},
$$
and
$$
\sup_{x' \in B_{X^*}} \Big( \int_\mathbb N |R(gf,x')|^{\frac{p}{1-\sigma}} \, d c \Big)  = \sup_{x' \in B_{X^*}} \Big( \sum_{i=1}^n \big( | \langle g(i) x_i, x' \rangle|^{1-\sigma} \| g(i)x_i \big\|^\sigma \big)^{\frac{p}{1-\sigma}}
 \Big).
$$
If $u$ is $(q,p,\sigma)$-absolutely continuous then (\ref{pq}) occurs for all $x_1,\ldots,x_n\in X$. In particular for $ g(1)x_1,\ldots,g(n)x_n$, for all $x_1,\ldots,x_n\in X$ and all $g\in \ell_q$. Hence, $u$ is $(\frac{q}{1-\sigma},\frac{p}{1-\sigma})$-$RS$ summing.
Reciprocally, if $u:X \to Y$ is  $(\frac{q}{1-\sigma},\frac{p}{1-\sigma},\sigma)$-RS summing for the given definitions of $R$ and $S$ then, for arbitrary $x_1,\ldots,x_n\in X$ take $g=\chi_{\{1,\ldots,n\}}\in \mathcal L_q(c)$. It follows that $u$ is $(q,p,\sigma)$-absolutely continuous. }
\end{proof}

\begin{corollary} \label{inclu2}
Let $1 \le p_1 \le p_2 < \infty$, $ 1 \le q_1 \le q_2 < \infty$ and $0 \le \sigma < 1$.  Suppose also that
\[
\frac{1}{p_{1}}-\frac{1}{p_{2}}\leq\frac{1}{q_{1}}-\frac{1}{q_{2}}.
\]
Then
$$
\Pi_{q_1,p_1}^\sigma(X,Y) \subseteq \Pi_{q_2,p_2}^\sigma(X,Y).
$$
\end{corollary}
\begin{proof}
This is a direct consequence of Proposition \ref{inclu}. Indeed, take $R$ and $S$ as said above. Then the elements of
{$\Pi_{q_i,p_i}^\sigma(X,Y)$
coincide with the ones of $RS(\frac{q_i}{1-\sigma}, \frac{p_i}{1-\sigma})$, $ i=1,2$,} due to the definitions of $R$ and $S$. Clearly, $\frac{q_1}{1-\sigma} \le \frac{q_2}{1-\sigma}$
and $\frac{p_1}{1-\sigma} \le \frac{p_2}{1-\sigma}$ if and only if $q_1 \le q_2$ and $p_1 \le p_2$, and
$$
\frac{1-\sigma}{p_1} - \frac{1-\sigma}{p_2}   \le \frac{1-\sigma}{q_1} - \frac{1-\sigma}{q_2}
$$
if and only if $\frac{1}{p_{1}}-\frac{1}{p_{2}}\leq\frac{1}{q_{1}}-\frac{1}{q_{2}}.$ This gives the result.
\end{proof}

\begin{remark}
The same definition and extensions can be done for two different interpolation parameters $\sigma_1$ and $\sigma_2$ in the left and right hand sides of the inequalities for getting a more general definition of
absolutely continuous operators depending on four parameters. These classes have not being studied yet, but the expected results would be similar to the ones obtained in this section.
\end{remark}

\section{Absolutely continuous operators and  integrability of strongly measurable functions}

First, it is important to remark that given a continuous linear operator $u:X\to Y$, the map
 $\tilde u:{\mathcal P}_1(\mu,X) \to {\mathcal P}_1(\mu,Y) $ given by $\tilde u(f)=u\circ f$,
 is well-defined and continuous, according to the proof of the theorem in \cite{Di}. In all this section  $\mu$ will be a \textit{finite} measure. The following result establishes the link
among summability of the operators and integrability of the corresponding vector valued functions.

\vspace{0.5cm}

\textbf{Theorem} (\textit{Theorem 5 in \cite{perusan}})
\textit{Let $0 \le \sigma \le 1$ and let $\mu$ be a finite measure.
An operator  $u:X \to Y$ is $(1,\sigma)$-absolutely continuous if and only if the composition operator $\tilde{u}:{\mathcal P}_{1/(1-\sigma)}^\sigma(\mu,X)\to {\mathcal B}_{1/(1-\sigma)}(\mu,Y)$ given by $\tilde u(f):=u\circ f$ is well defined and continuous. In this case, }
$$
\pi^\sigma_{1} (u) = \| \tilde{u} \|.
$$

Together with the inclusion result shown in the previous section, this provides useful information about the relation among integrability of vector valued functions and summability properties of the operators.
Note that $\Pi^\sigma_{1}(X,Y)= \Pi^\sigma_{1,1}(X,Y)$.


\begin{corollary} \label{corun}
Let $0 \le \sigma \le 1$ and let $\mu$ be a non-atomic finite measure.
Let  $u:X \to Y$ be an operator and consider the composition operator
$$
\tilde{u}:{\mathcal P}_{1/(1-\sigma)}^\sigma (\mu,X)\to {\mathcal B}_{1/(1-\sigma)}(\mu,Y)
$$
given by $\tilde u(f):=u\circ f$.

\begin{itemize}
\item[(i)] If $\tilde u$ is continuous,
then $u \in  \Pi_{q,p}^\sigma(X,Y) $ for all  $1 \le p \le q < \infty$.

\item[(ii)] {If $\Pi_{\frac{1}{1-\sigma},1}(X,Y)=\Pi_1^\sigma (X,Y)$ then, $ \Pi_{\frac{1}{1-\sigma},1}(X,Y)\subset \Pi_{q,p}^\sigma(X,Y) $ for all  $1 \le p \le q < \infty$. In particular,  $\Pi_{\frac{1}{1-\sigma},1}(C(K),Y) \subset \Pi_{q,p}^\sigma(C(K),Y) $ for all  $1 \le p \le q < \infty.$}


\item[(iii)]  If $1\leq s \leq {1}/{(1-\sigma)}<\infty$ and ${\mathcal L}(X,Y)=\Pi_{s}(X,Y),$ then $\tilde{u}$ is continuous and ${\mathcal L}(X,Y)=\Pi_{p,q}^\sigma(X,Y)$ for $1 \le p \le q < \infty$.
For instance, if $X$ is an $L^\infty$-space or an $L^1$-space and $Y$ is a Hilbert space, we get the result for $\sigma \ge 1/2$ and $s=2$.
\end{itemize}
\end{corollary}
\begin{proof}
For the proof of (i), just take into {account  that for $p_1=1=q_1$, $q=q_2$ and $p=p_2$, we have that
$$
1- \frac{1}{p} \le 1- \frac{1}{q},
$$
whenever $p \le q$.   Corollary \ref{inclu} gives that $ \Pi_1^\sigma(X,Y)\subset \Pi_{p,q}^\sigma(X,Y)$.} Now a direct application of the Theorem (\textit{Theorem 5 in \cite{perusan}}) gives the result.

(ii)
We can prove this directly as a consequence of the inclusion result given by Corollary \ref{inclu2}. Alternatively,
by Corollary 9 in \cite{perusan},{ for
$1\leq s=1/(1-\sigma)<\infty$ and $\sigma=1/s'$,} we have that if $\Pi_{s,1}(X,Y)=\Pi_1^\sigma (X,Y),$
then $u\in \Pi_{s,1}(X,Y)$ if and only if  $\tilde u:{\mathcal P}_s^\sigma(\mu,X)\to {\mathcal B}_s(\mu,Y)$ is well-defined and continuous.
Consequently, by (i)  we directly obtain  that  $u \in \Pi_{q,p}^\sigma(X,Y)$.

The case $X=C(K)$ is given by an application of the results of \cite{LS97}.

(iii)
Suppose that
 $1\leq s\leq \frac{1}{1-\sigma}<\infty$. Assume that ${\mathcal L}(X,Y)=\Pi_{s}(X,Y).$ Then by Corollary  10 in \cite{perusan}
$\tilde u:{\mathcal S}_\frac{1}{1-\sigma}^\sigma(\mu,X)\to {\mathcal B}_\frac{1}{1-\sigma}(\mu,Y)$
 is well-defined and continuous for any $u\in {\mathcal L}(X,Y)$. The inclusion given by Corollary \ref{inclu2} provides the second part of (iii).
The last result is a consequence of the sometimes called ``little Grothendieck theorem", see \S 11.11 in \cite{DF92}.
\end{proof}

\vspace{1cm}

Let us finish with a concrete application for $L^1$-spaces. It can be proved as a direct consequence of Corollary 15 in \cite{perusan} and Corollary \ref{corun}(i).
Recall first that  we say that a Banach space $E$ is $(\sigma,p)$-Hilbertian ---see the definition in   \cite{Matt87}---,
if there is an interpolation pair $(H,F)$, where $H$ is a Hilbert space and $F$ is a Banach space, in such a way that  $E$ coincides isomorphically with
 the real interpolation space $(H,F)_{\sigma,p}$, $1 \le p < \infty$ and $0 \le \sigma < 1$.
In the same way, it is said that $E$ is
$\sigma$-Hilbertian   if it is isomorphic to a complex interpolation space $E=[H,F]_\sigma$.

Consider  an ${L}^1$-space $L^1$ and a non-atomic finite positive measure $\mu$.
Let $0 \le \sigma <1$ and consider an operator $u:L^1 \to E$. Then for all $1 \le p \le q < \infty,$

\begin{itemize}
\item[(1)] if $E$ is a quotient of an $L^\infty$-space having cotype smaller that $\frac{2}{1-\sigma},$ then $u \in \Pi_{q,p}^\sigma(L^1,E).$

\item[(2)] If $E= L^r$ for $2 \le r < \frac{2}{1-\sigma}, $ then $u \in \Pi_{q,p}^\sigma(L^1,E).$

\item[(3)] If $E$ is a $(\sigma,2)$-Hilbertian space, then $u \in \Pi_{q,p}^\sigma(L^1,E).$

\item[(4)] If $E$ is a $\sigma$-Hilbertian space, then $u \in \Pi_{q,p}^\sigma(L^1,E).$

\item[(5)] If $E$ is a Lorentz space $L_{r,s}$ for $\frac{2}{1+\sigma} < r$ and $s < \frac{2}{1-\sigma}$, then $u \in \Pi_{q,p}^\sigma(L^1,E).$

\end{itemize}

\end{document}